\documentclass[11pt,a4paper]{amsart}
\input xy
\xyoption{all}
\usepackage{amssymb}

\DeclareMathOperator{\Aut}{Aut}
\DeclareMathOperator{\Out}{Out}
\DeclareMathOperator{\Epi}{Epi}
\DeclareMathOperator{\PSL}{PSL}

\DeclareMathOperator{\SL}{SL}

\DeclareMathOperator{\Inn}{Inn}
\DeclareMathOperator{\Sym}{Sym}
\DeclareMathOperator{\Hom}{Hom}
\DeclareMathOperator{\Mod}{Mod}

\newcommand{\g}{\underline{g}}
\newcommand{\x}{\underline{x}}

\newcommand{\h}{\underline{h}}

\newcommand{\trivgp}{\langle e \rangle}
\newcommand{\sgg}{\left \langle \underline{g} \right \rangle}

\newcommand{\id}{{\bf{Id}}}
\def\cprime{$'$}

\newtheorem{theorem}{Theorem}
\newtheorem{lemma}[theorem]{Lemma}
\newtheorem{proposition}[theorem]{Proposition}
\newtheorem{corollary}[theorem]{Corollary}

\theoremstyle{definition}
\newtheorem{question}[theorem]{Question}
\newtheorem{definition}[theorem]{Definition}
\newtheorem{note}[theorem]{Note}
\newtheorem{remark}[theorem]{Remark}

\newcommand{\N}{{\mathbf{N}}}
\newcommand{\Z}{{\mathbf{Z}}}

\newcommand{\F}{{\mathbf{F}}}

\newcommand{\arrow}{\rightarrow}

\title{Highly transitive actions of ${\mathbf{\operatorname{Out}(F_n)}}$}

\author{Shelly Garion}
\thanks{The first author was supported by a European Postdoctoral Fellowship (EPDI) }
\email{shellyg@ihes.fr}
\address{Institut des Hautes {\'E}tudes Scientifiques, route de Chartres, 91440 Bures-sur-Yvette, France.}

\author{Yair Glasner}
\thanks{The second author was partially supported by ISF grant 888/07. }
\email{yairgl@math.bgu.ac.il}
\address{Department of Mathematics, Ben-Gurion University of The Negev, Be'er Sheva, Israel.}


\begin{document}
\bibliographystyle{alpha}

\begin{abstract}
An action of a group on a set is called $k${\it{-transitive}} if it is transitive on ordered $k$-tuples and {\it{highly transitive}} if it is $k$-transitive for every $k$. We show that for $n \ge 4$ the group $\Out(F_n) = \Aut(F_n) / \Inn(F_n)$ admits a faithful highly transitive action on a countable set.
%
\end{abstract}

\maketitle

\section{introduction}
\subsection{Highly transitive actions. }
The group $\Out(F_n) = \Aut(F_n)/\Inn(F_n)$ of outer automorphisms of the free group attracted much attention in the last couple of decades. The theory that is developed around this group runs parallel to that of the mapping class group of a surface $\Mod(\Sigma_g) = \Out(\pi_1(\Sigma_g))$ and the special linear group $\SL_n(\Z)  = \Out(\Z^n)$.  The questions that are asked about the first two groups are often motivated by the more classical theory of the arithmetic group $\SL_n(\Z)$ but sometimes the answers exhibit new and interesting phenomena. In the lowest non-trivial case these three families  coincide $\Out(F_2) = \Mod(\Sigma_1) = \SL_2(\Z)$ and  then they ramify in different directions. Moreover, for large values of $n$ all three theories exhibit interesting ``higher rank'' phenomena that are not shared by the group $\SL_2(\Z)$. 

Due to the efforts of many mathematicians, notably Margulis, we can exhibit today an intricate and beautiful structure theory for the arithmetic groups $\SL_n(\Z)$, for $n \ge 3$. This theory provides a complete understanding of many properties of these groups, including their normal subgroups, finite index subgroups and finite dimensional linear representations. More generally it is safe to say that we have a complete understanding of representations of $\SL_n(\Z)$ into locally compact groups, summarizing all of the above. In particular it was shown by Venkataramana \cite{Venkataramana:LC_completions} and upcoming paper of Willis-Shalom that every representation of $\rho: \SL_n(\Z) \arrow G$ into a locally compact group has either a pre-compact or a discrete image. 

There is a lot of information pertaining to representations of $\SL_n(\Z)$ into polish groups that are not locally compact; including unitary representation (see for example \cite{BHP:Kazhdan_T_book}), action on manifolds (see \cite{Zimmer_Morris:book}) and infinite permutation representations. In this paper we are interested in the latter theory, namely representations of a group $\Gamma$ into  $\Sym(\Z)$ - the full (polish) group of permutations of $\Z$, considered here as a countable set. From a different perspective this is also the study of subgroups of $\Gamma$, via the well known correspondence between transitive permutation representations and (conjugacy classes of) subgroups.  At the moment we have a very limited understanding of permutation representations of $\SL_n(\Z)$. Primitive permutation representations for $\SL_n(\Z)$ and many other finitely generated linear groups, were constructed by Margulis and Soifer \cite{MS:first,MS:Maximal_Subgroups}.  This work was later extended in \cite{GG:Primitive} to construct {\it{faithful}} representations of many countable linear groups that are not necessarily finitely generated. From the point of view of permutation representations however, primitivity is a rather weak notion, for example a primitive permutation representation $\rho: \Gamma \arrow \Sym(\Z)$ might still have a discrete image. In particular it is an open question whether the analogue of the Venkataramana and Willis-Shalom theorem is valid in this setting, namely if there exists permutation representations $\rho: \SL_n(\Z) \arrow \Sym(\Z)$ whose image is neither discrete nor pre-compact. 

In this paper we solve the above mentioned question for the group $\Gamma = \Out(F_n)$ for every $n \ge 4$, by exhibiting permutation representations that are as far from being discrete as possible:
\begin{theorem}
For every $n \ge 4$ there exists a faithful permutation representation $\rho: \Out(F_n) \arrow \Sym(\Z)$ with a dense image. Or in other words a permutation representation that is $k$-transitive for every $k$. 
\end{theorem}
\noindent This theorem is stated again below as Theorem \ref{thm:Main} with a more detailed description of the action. It is interesting to note that the answer to the above mentioned question is solved for $\Out(F_n)$ before it is solved for $\SL_n(\Z)$ or for $\Mod(\Sigma_g)$. In fact it might very well be the case that $\SL_n(\Z)$ fails to admit such a highly transitive permutation representation. 

Let us just mention that very few examples are known of countable groups that admit highly transitive action on a set.  Examples include non abelian free groups (see \cite{MD_free_HT,Dixon:free_HT} as well as fundamental groups of surfaces of genus at least $2$ (see \cite{Kitroser:HT}), the group of finitely supported permutations $\Sym_C(\Z)$ and, of course, any subgroup of $\Sym(\Z)$ containing one of these examples.

\subsection{Notation}
Throughout the paper $G$ will denote a finitely generated simple group, with $d = d(G)$ its minimal number of generators. Given an n-tuple $\g = (g_1,g_2,\ldots, g_n) \in G^n$ we denote by $\sgg = \langle g_1,g_2,\ldots, g_n \rangle < G$ the subgroup generated by it. Let $F_n = \langle x_1,x_2,\ldots,x_n \rangle$ be the nonabelian free group on $n$ generators. There is a natural identification of $G^n$  with $\Hom(F_n,G)$; associating the $n$-tuple $\g$ with the homomorphism
\begin{eqnarray*}
\alpha_{\g}: F_n & \arrow & G \\
\alpha_{\g}(x_i) & \mapsto & g_i.
\end{eqnarray*}
We will use $\Hom(F_n,G)$ and $G^n$ interchangeably. For example we will identify the set of epimorphisms $\Epi(F_n,G)$ with the set $$V_n(G) = \{ \g \in G^n : \sgg = G \},$$ of all generating $n$-tuples.

The group $\Aut(G)$ acts on $\Hom(F_n,G)$ (from the left) by post-composition and $\Aut(F_n)$ acts on $\Hom(F_n,G)$ (from the right) by pre-composition. One easily verifies that in the corresponding action of $\Aut(F_n)$ on $G^n$ the standard Nielsen transformations act as follows:
\begin{align*}
    R_{i,j}^{\pm} &:  (g_1,\ldots,g_i,\ldots,g_n) \rightarrow
    (g_1,\ldots,g_i\cdot g_j^{\pm 1},\ldots,g_n), \\
    L_{i,j}^{\pm} &:  (g_1,\ldots,g_i,\ldots,g_n) \rightarrow
    (g_1,\ldots,g_j^{\pm 1} \cdot g_i,\ldots,g_n), \\
    P_{i,j} &:  (g_1,\ldots,g_i,\ldots,g_j,\ldots,g_n) \rightarrow
    (g_1,\ldots,g_j,\ldots,g_i,\ldots,g_n),  \\
    I_{i} &:  (g_1,\ldots,g_i,\ldots,g_n) \rightarrow
    (g_1,\ldots,g_i^{-1},\ldots,g_n).
\end{align*}
We denote by $\bar V_n(G)$ the set of all $\Aut(G)$-orbits on $V_n(G)$. Given $\g \in V_n(G)$ we denote by $\left[\g\right]$ the corresponding equivalence class in $\bar V_n(G)$. Since the $\Aut(F_n)$ action preserves $V_n(G)$ and commutes with the $\Aut(G)$ action it descends to an action of $\Gamma := \Out(F_n) = \Aut(F_n) / \Inn(F_n)$ on $\bar V_n(G)$. 

\subsection{The main theorem}
A Tarski monster group is a noncyclic group $G$, all of whose
proper subgroups are cyclic. It was shown by A. Yu. Ol{\cprime}shanski{\u\i} that for every large enough prime $p$ there exist uncountably many Tarski monsters all of whose subgroups are isomorphic to $\Z/p\Z$, as well as Tarski monsters all of whose subgroups are infinite cyclic (see \cite{Ol:Tarski_Monsters}).

Establishing the existence of infinite Tarski monsters is difficult, but once such a group is given many structural results follow directly from the definition. A Tarski monster $G$ is necessarily simple. For every $n \ge 2$, the collection of generating $n$-tuples is given by
$$V_n(G) =  G^n \setminus \left \{\g \in G^n \ | \ \sgg {\text{ is cyclic}} \right \}.$$

\begin{definition}
An action of a group on a set $\Gamma \curvearrowright \Omega$ is called {\it{$k$-transitive}} if it is transitive on ordered $k$-tuples of distinct points. It is called {\it{highly transitive}} if it is $k$-transitive for every $k \in \N$.
\end{definition}
\begin{definition}
A group $G$ is said to satisfy a {\it{group law}}, if there exists some $m$ and some non-trivial word in the free group $w \in F_m$ such that $$w(\g) = 1, \qquad \forall \g \in G^m.$$
\end{definition}
\noindent
For example every abelian group satisfies the law given by the commutator $[x,y] \in F_2$ and every group of exponent $p$ satisfies the word $x^p \in \Z$. The goal of this paper is to prove the following
\begin{theorem} \label{thm:Main}
Let $G$ be a Tarski monster and $n \ge 4$ then the action of $\Gamma = \Out(F_n)$ on the character variety $\bar{V}_n(G)$ is highly transitive. Moreover, this action is faithful if and only if $G$ does not satisfy a group law.
\end{theorem}

Clearly every Tarski monster of finite exponent satisfies a group law. But there are torsion free Tarski monsters that do not admit a group law. Indeed it was observed by Zusmanovich \cite[Theorem 5.1]{Zus:Tarski_law} that Tarski monsters fail to satisfy a group law if and only if they admit finite generating sets whose corresponding Cayley graphs have arbitrarily large girth. The existence of such Tarski monsters was established by Ol{\cprime}shanski{\u\i}; the argument based on \cite[Corollary 1]{OL:residualing} is described by Zusmanovich [loc. cit.]. A complete proof can be found in \cite{OOS:Lacunary_H_groups}.   
\begin{corollary}
For $n \ge 4$ let $\Gamma_n = \Out(F_n)$ and let $\Gamma_n \curvearrowright \Omega$ the faithful highly transitive action constructed above. Let $\Delta  < \Gamma_n$ be any subgroup containing a non-trivial subnormal subgroup. Then the following holds
\begin{enumerate}
\item \label{itm:subnor} The action of $\Delta$ on $\Omega$ is highly transitive, 
\item \label{itm:max} $\Delta$ contains an infinite index maximal subgroup,
\item \label{itm:gplaw} $\Delta$ does not satisfy any group law, in particular it cannot be finite, abelian, or even virtually solvable. 
\item \label{itm:inter} If $\Delta' < \Gamma$ is another group containing a non-trivial subnormal subgroup then $\Delta \cap \Delta' \ne \trivgp$. 
\end{enumerate}
\end{corollary}
\begin{proof}
Saying that $\Gamma_n \curvearrowright \Omega$ is highly transitive is equivalent to saying that the corresponding embedding $\Gamma_n < \Sym(\Omega)$ has a dense image; with respect to the pointwise convergence topology on $\Sym(\Omega)$. But then $\overline{N} \lhd \lhd \overline{\Gamma_n} = \Sym(\Omega)$ for any subnormal subgroup $N \lhd \lhd \Gamma_n$ and since $\Sym(\Omega)$ is topologically simple $N$ is also dense, proving (\ref{itm:subnor}). Now (\ref{itm:max}) follows, since a highly transitive action is automatically primitive so $\Delta_\omega < \Delta$ is a maximal subgroup for every $\omega \in \Omega$. Since $\Sym(\Omega)$ contains a free subgroup it cannot satisfy any group law itself and this is automatically inherited by any dense subgroup, which establishes (\ref{itm:gplaw}). Finally for (\ref{itm:inter}) assume by way of contradiction that $N, N' \lhd \lhd \Gamma_n$ are two non-trivial subnormal subgroups that intersect trivially. We claim that there are two (possibly equal) non-trivial subnormal subgroups $M,M'$ that commute. Indeed let $N = N_l \lhd N_{l-1} \lhd \ldots \lhd N_0 = \Gamma_n$ and assume first that $N' \lhd \Gamma_n$ is normal. If $j$ is the first index such that $N' \cap N_j = \trivgp$ we set $M = N_j$ and $M' = N' \cap N_{j-1}$. The argument is concluded by induction on the minimal length of a subnormal series for $N'$.  
Now since $M,M'$ are both dense in the topology induced from $\Sym(\Omega)$ it turns out that $\Sym(\Omega)$ is abelian which is absurd. 
\end{proof}
Item (\ref{itm:max}) above can be viewed as an analogue for $\Out(F_n)$ of the theorems due to Margulis and So\u{\i}fer \cite{MS:first,MS:Maximal} and to Ivanov \cite[Theorem 5]{Ivanov:MCG} concerning the existence of infinite index maximal subgroups (i.e. of the existence of primitive actions on countable sets) . Margulis and So\u{\i}fer prove that a finitely generated linear group admits an infinite index maximal subgroup if and only if it is not virtually solvable. Ivanov's theorem states that finitely generated subgroup of the mapping class group of a surface admits a maximal subgroup of infinite index if and only if it is not virtually abelian. Both theorems were then generalized in \cite{GG:Primitive} to general countable subgroups. On a certain level our current result is much stronger because highly transitive actions are rarer than primitive actions. On the other hand our current method is restricted to very special classes of subgroups of $\Out(F_n)$. 
 
\subsection{Gilman's work on the Wiegold conjecture}
Transitivity of the action of $\Gamma \curvearrowright \bar{V}_n(G)$, for various groups $G$ was extensively studied in various different settings in the last few decades. We refer the readers to a comprehensive survey article on this subject by Alex Lubotzky \cite{Lub:Aut_Fn}. In particular it is conjectured by Wiegold that the action of $\Out(F_3) \curvearrowright \bar{V}_3(G)$ is transitive for every finite simple group $G$.

Of particular interest from our point of view is the work of Gilman \cite{Gilman:SL2p} who proved the Wiegold conjecture in the case $G = \PSL_2(\F_p)$ for every prime $p \ge 5$ and $n \ge 3$. Gilman showed, in fact, that the image of $\Out(F_n)$ in $\Sym(\bar{V}_n(\PSL_2(\F_p))$ is either the full symmetric group, or the alternating group. Thus proving a much stronger statement.

Gilman has further proved that if $G$ is a finite simple nonabelian group and $n \geq 4$ then $\Out(F_n)$ acts as a symmetric or alternating group on at least one of its orbits in $\bar{V}_n(G)$.
This result was extended to $n=3$ by Evans \cite{Evans:Wiegold_2}.

The current paper grew out of an attempt to find an infinite setting in which Gilman's proof can be implemented. Technically there is one qualitative difference between the finite and the infinite case. In the former in order to prove that a finite permutation group contains the alternating group it is sufficient to prove primitivity and then establish the existence of one long cycle. In the latter $k$-transitivity has to be verified directly, by induction, for every $k \in \N$.

\subsection*{Acknowledgment}
We would like to thank Dawid Kielak for his helpful comments. We would also like to thank the referee for reading so carefully the original manuscript and providing corrections as well as very helpful information pertaining to the structure of Tarski Monsters. 

This joint work was initiated following the Lis Gaines Workshop ``Action of $\Aut(F_n)$ on representation varieties'', that was held in January 2009 at Sde-Boker, Israel.


\section{Preliminary results}

\subsection{Generation of powers} \label{sec:gen_power}
We denote by $M_{n,k}(G)$ the collection of $n \times k$ matrices with entries taken from the group $G$. Let
\begin{equation} \label{eqn:matrix}
A = \begin{pmatrix}
g_{1}^{1}  & \ldots & g_{1}^{k} \\
g_{2}^{1} & \ddots  & g_{2}^{k} \\
\vdots & &  \vdots \\
g_{n}^{1} & \ldots & g_{n}^{k}
\end{pmatrix} =
\begin{pmatrix}
-  & \g_1 & - \\
- & \g_2 & - \\
  & \vdots  &  \\
- & \g_n & - \\
\end{pmatrix} =
\begin{pmatrix}
|  &  & | \\
\g^{1} & \hdots  & \g^{k} \\
| & & |
\end{pmatrix},
\end{equation}
be such a matrix, with $\g_1,\ldots,\g_n \in G^k$ and $\g^1,\ldots \g^k \in G^n$ denoting the corresponding row and column vectors respectively.
The following result was originally used by P. Hall \cite{Hall:gen} in the realm of finite simple groups. Compare for example \cite[Prop. 6]{KL:prob_gen}.
For the convenience of the reader we add a proof.

\begin{proposition} \label{prop:gen_prod}
Let $G$ be a nonabelian simple finitely generated group, $k \ge
d(G)$ and $A \in M_{n,k}(G)$ a matrix as above, all of whose columns
generate: $\g^i \in V_n(G), \ \ \forall 1 \le i \le k$. Denote by
$\alpha^1,\ldots, \alpha^k \in \Hom(F_n,G)$ the corresponding
epimorphisms, by $K^1,\ldots,K^k$ their respective kernels and by $A
= (\alpha^1, \alpha^2,\ldots,\alpha^k)$ the combined homomorphism
defined by
\begin{eqnarray*}
A  : F_n & \arrow & G^{k} \\
w & \mapsto & (w(\g^1),w(\g^2),\ldots,w(\g^k)).
\end{eqnarray*}
Then the following are equivalent

\begin{enumerate}
\item \label{itm:inVbar} $\{[\g^i] \ | \ 1 \le i \le k \}$ are all different as elements of $\bar{V}_n(G)$.
\item \label{itm:kernels} $\alpha^j \left(\cap_{i \ne j} K^i \right) = G$, for every $1 \le j \le k$.
\item \label{itm:products} $A: F_n \arrow G^k$ is surjective.
\end{enumerate}
\end{proposition}
\begin{proof}
(\ref{itm:kernels}) $\Longrightarrow$ (\ref{itm:products}): is obvious. \\
(\ref{itm:products}) $\Longrightarrow$ (\ref{itm:inVbar}): Assume (\ref{itm:inVbar}) fails. By definition this means that there is some $\sigma \in \Aut(G)$ and indices $i,j$ such that $\sigma \circ \alpha^j = \alpha^i$. But then $A(F_n) < \{(g_1,\ldots,g_k) \ | \ g_i = \sigma(g_j)\}$ contradicting (\ref{itm:products}). \\
(\ref{itm:inVbar}) $\Longrightarrow$ (\ref{itm:kernels}). By
symmetry we argue for $j = k$. Since $\alpha^k : F_n \arrow G$ is
surjective and $\cap_{i=1}^{k-1}(K^i) \lhd F_n$ is normal, it
follows that $\alpha^k \left(\cap_{i=1}^{k-1}(K^i) \right) \lhd G$.
By simplicity of $G$ we need only rule out the possibility that the
latter group is trivial. If $k=2$ this implies that $K^1 < K^2$ and
hence there is a commutative diagram:
$$\xymatrix{
F_n \ar[r]^{\alpha^1} \ar[dr]^{\alpha^2}& G \ar[d]^{\eta} \\
        & G
}$$
Since $G$ is simple $\eta$ must be an isomorphism, contradicting (\ref{itm:inVbar}).

We proceed by induction on $k$. Recall that by assumption, $G$ is
nonabelian and let $x,y \in G$ be two elements with $[x,y] \ne e$.
By our induction assumption we can find $\tilde{x}, \tilde{y} \in
F_n$ such that
\begin{eqnarray*}
A(\tilde{x}) & = & (*,e,e,\ldots,e,x), \\
A(\tilde{y}) & = & (e,*,e,\ldots,e,y).
\end{eqnarray*}
It follows that $A \left([\tilde{x},\tilde{y}] \right) =
(e,e,\ldots,e,[x,y])$, contradicting the triviality of $\alpha^k
\left(\cap_{i=1}^{k-1}(K^i) \right)$.
\end{proof}

\begin{note}
We assumed implicitly in the above proposition that $k$ is finite.
If $k = \infty$ is infinite it is no longer possible for $A$ to be surjective, but the above proposition remains correct upon replacing (\ref{itm:products}) by the assumption that $\overline{A(F_n)} = G^{\times \infty}$ where the closure is taken with respect to the product topology.
\end{note}

\subsection{Powers of Tarski monsters.}
Denote by $d(G)$ the minimal number of generators of a group $G$. It was shown by Wiegold and Wilson in \cite[Theorem 4.3]{WW:Growth_seq} that $d(G) \le d(G^k) \le d(G)+1, \quad \forall k \in \N$ for every finitely generated infinite simple group $G$. In the case where $G$ is a torsion free Tarski Monster it was further shown by Wiegold in \cite[Theorem 2]{Wiegold:square} that $d(G^k) = d(G) = 2$. Since Wiegold argument in the latter paper is essential to our argument and the paper itself was somewhat difficult to obtain we repeat here Wiegold's argument in our own terminology. 
\begin{theorem} (Wiegold) \label{thm:Wieg}
Let $G$ be a Tarski monster then $2 \le d(G^k) \le 3, \quad \forall k \in \N$. Moreover if $G$ is torsion free then $d(G^k) =2 \quad \forall k \in \N$.
\end{theorem}
\begin{proof}
Given $k \in \N$ we have to establish the existence of a matrix $A \in M_{3 \times k}(G)$ satisfying the equivalent conditions of Proposition \ref{prop:gen_prod}.

Indeed let $a,b \in G$ be a generating tuple, and consider the
matrix:
$$A = \begin{pmatrix}
a & a & \ldots & a \\
b & b & \ldots & b \\
c_1& c_2 & \ldots & c_k
\end{pmatrix} \in M_{3 \times k}(G)$$
Clearly every column generates. But if two columns, say $\g^1$ and $\g^2$ are in the same $\Aut(G)$ orbit then we have an automorphism $\sigma \in \Aut(G)$ such that $\sigma(a) = a; \sigma(b) = b; \sigma(c_1) = c_2$. Since $a,b$ generate this implies that $\sigma = \id$ and $c_1 = c_2$. Thus condition (\ref{itm:inVbar}) of Proposition \ref{prop:gen_prod} will be satisfied if all $c_i$ are different.

Assume now that $G$ is torsion free. Since $G$ is finitely generated it contains a proper maximal subgroup $\langle a \rangle$, let $b \in G$ be any element that fails to commute with $a$. By maximality of $\langle a \rangle$ if $[a,b^{-1}ab] = 1$ then $b^{-1}ab = a^m$ for some $m$ and hence $b \in N_G (\langle a \rangle) = Z_G(\langle a \rangle)$, contradicting our choice of $b$. Thus $G = \langle a , b^{-1}ab \rangle$. 

Consider a matrix $A \in M_{2 \times k}(G)$ given by 
$$A = \begin{pmatrix}
a & a & \ldots & a \\
ab & a^{2}b & \ldots & a^{k}b
\end{pmatrix}$$
Clearly every column generates. Moreover if $\sigma \in \Aut(G)$ is such that $a^{\sigma} = a$ and $(a^lb)^{\sigma} = (a^mb)$ then $(b^{-1}ab)^{\sigma} = ((a^{m-l}b)^{-1}a(a^{m-l}b)) = b^{-1}ab$. Since $\langle a, b^{-1}ab \rangle = G$ and both these elements are fixed by $\sigma$ the latter automorphism must be the identity. Hence condition (\ref{itm:inVbar}) of Proposition \ref{prop:gen_prod} is satisfied and the proof is complete. \end{proof}

\subsection{Spread}
\begin{definition}
We say that a $2$-generated group $G$ has {\it{spread greater or equal to $k$}} if for every $\g = (g_1,g_2, \ldots,g_k) \in G^k$ there exists some $h \in G$ such that $$G = \langle h, g_i \rangle \quad \forall 1 \le i \le k. $$
\end{definition}
\begin{lemma} \label{lem:spread}
A Tarski monster group $G$ has spread greater or equal to $k$ for every $k \in \N$.
\end{lemma}
\begin{proof}
Let $k \in \N$ and $\g = (g_1,g_2,\ldots,g_k) \in G^k$ be as above.
Since $G$ is finitely generated, it follows from Zorn's lemma that
every $g_i$ is contained in a maximal proper subgroup $g_i \in H_i
\lneqq G$. By the definition of a Tarski monster every $H_i$ is
cyclic and hence of infinite index. But an infinite group is never a
union of a finite number of subgroups of infinite index. Thus, any
$h \in G \setminus \cup_{i = 1}^k H_i$ will satisfy the condition
required in the definition of the spread.
\end{proof}

\subsection{Stronger generation properties}
In order to prove Theorem \ref{thm:Wieg} we constructed a matrix $3 \times k$ all of whose columns represent different elements of $\bar{V}_n(G)$. In the sequel we will need a matrix satisfying a stronger condition, which is somewhat technical but useful. \begin{lemma} \label{lem:tech_3_gen}
For any $k, n \in \N$ with $n \ge 4$ there exists a matrix $A \in M_{n \times k}(G)$, as in Equation (\ref{eqn:matrix}) in Section \ref{sec:gen_power}, with the following properties:
\begin{enumerate}
\item \label{itm:pair} Every pair of entries generates $G$: $$(i,j) \ne (l,m) \Rightarrow \langle g^{i}_j, g^{l}_{m} \rangle = G,$$
\item \label{itm:three_row} Every three rows generate $G^k$: $$1 \le i < j < l  \le n \Rightarrow \langle \g_i,\g_j,\g_l \rangle = G^k.$$
\item \label{itm:config} The following configuration will never appear as a $4 \times 4$ minor of the matrix, for any $\sigma,\tau,\eta, \theta  \in \Aut(G)$ and $p,q,r,s \in G$.
\begin{equation*} \label{eqn:config}
\begin{pmatrix}
\cdot              & p^{\tau}  & p^{\eta} & p^{\theta} \\
q^{\sigma}     & \cdot       & q^{\eta} & q^{\theta} \\
r^{\sigma}      & r^{\tau}   & \cdot      & r^{\theta} \\
s^{\sigma}     & s^{\tau}   & s^{\eta} & \cdot
\end{pmatrix}
\end{equation*}
\end{enumerate}
\end{lemma}
\begin{proof}

We construct the matrix entries one by one according to the following order
$$g^{1}_1, g^{1}_2, g^{1}_3, \ldots, g^{2}_1,g^{2}_2, \ldots, g^{k}_{n-1},g^{k}_n,$$
namely column by column. Making sure in the process that all three desired properties hold:

\noindent {\bf{(\ref{itm:pair}) Every pair of elements generates:}}
Upon adding the element $g^m_l$ one can make sure that it generates $G$ with every previous entry by Lemma \ref{lem:spread}. In fact the proof of that lemma makes it clear that there are infinitely many possible choices of an element that will satisfy this condition.  Thus we can guarantee the validity of (\ref{itm:pair}) even if we require later in the proof to exclude finitely many possibilities at every stage.

\noindent {\bf{General discussion:}}
 Assuming from now on that (\ref{itm:pair}) is indeed satisfied we notice that for a given $2 \times 2$ minor
$$\begin{pmatrix}
g^{i}_s & g^{j}_s \\
g^{i}_t & g^{j}_t
\end{pmatrix} $$
there is at most one $\sigma \in \Aut(G)$ such that $g^{j}_s =
\left(g^{i}_s\right)^{\sigma} $ and $g^{j}_t =
\left(g^{i}_t\right)^{\sigma}$. If this holds for a given $1 \le i <
j \le k$ and for some choice of $1 \le s < t \le n$ we say that the
columns $i,j$ are {\it{$\sigma$-near}}. Next we extend the notion of
near columns to be an equivalence relation - declaring two columns
$(i,j)$ to be {\it{$\sigma$-related}} if there is a sequence of
distinct columns $i=i_0, i_1,i_2, \ldots ,i_r =j$ such that
$i_{\alpha}$ and $i_{\alpha + 1}$ are $\sigma_{\alpha}$-near and
$\sigma = \sigma_0 \sigma_1 \sigma_2 \ldots \sigma_{r-1}$. We will
denote the set of all automorphisms relating two columns $1 \le i <
j \le k$ by:
$$\Xi_{i,j} := \left\{\sigma \in \Aut(G) \ | \ j {\text{ is $\sigma$-related to $i$}}\right\},$$
it is clear from the definition, and from the fact that any two matrix elements generate $G$ that this set is finite. Finally let us note that these relations can be defined even for matrices that are only partially defined, namely for matrices with some missing entries.

Assume that we are now adding the element $g^m_l$ namely that we have already constructed the matrix
$$\begin{pmatrix}
g^1_1     & g^2_1  & \ldots & g^{m-1}_1     & g^m_1  & \square\\
\vdots     & \vdots  & \ddots & \vdots            & \vdots & \square\\
g^1_{l-1} &  \vdots & \ldots & g^{m-1}_{l-1} & g^{m}_{l-1} & \square\\
\vdots     & \vdots  & \ddots & \vdots             & \square & \square \\
g^1_n     & g^2_n  & \ldots & g^{m-1}_n &   \square & \square \\
\end{pmatrix}$$
We choose an element $g^m_l$ which will satisfy condition (\ref{itm:pair}) while excluding the finite set of possibilities:
$$\bigcup_{i = 1}^{m-1}  \bigcup_{\sigma \in \Xi_{i,m}} \left( g^i_l \right)^{\sigma}$$

\noindent {\bf{(\ref{itm:three_row}) Every three rows generate $\bf{G^k}$}:}
By Proposition \ref{prop:gen_prod} all we have to do in order to prove that $G^k = \langle \g_i,\g_j,\g_l \rangle$, is to exclude the possibility that for some $1 \le s < m \le k$ and some $\sigma \in \Aut(G)$ we have
$$\begin{pmatrix}
g^s_i \\
g^s_j  \\
g^s_l
\end{pmatrix} =
\begin{pmatrix}
g^m_i \\
g^m_j \\
g^m_l
\end{pmatrix}^{\sigma},$$
where the automorphism applied to the vector just means that it is applied to each entry separately. But such a configuration is not possible by construction because at the time of the choice of the element $g^m_l$, the columns $s,m$ are already $\sigma$-related (in fact they are even $\sigma$-near) and thus the choice of $g^m_l = \left( g^s_l \right)^{\sigma}$ is ruled out.

\noindent {\bf{(\ref{itm:config}) Excluding cyclic configurations:}} A configuration such as the one appearing in (\ref{itm:config}) is excluded because at the time construction of the element labeled $r^{\theta}$ this choice is in fact invalid. Indeed at that time the first and the last column are already $\sigma^{-1} \theta = (\sigma^{-1} \eta) (\eta^{-1} \theta)$ related so that the choice $r^{\theta} = \left(r^{\sigma} \right) ^{\sigma^{-1} \theta}$ is invalid. Equivalently, the same choice can be ruled out, by the fact that the second and fourth columns are $\tau^{-1} \theta$-related.
\end{proof}

\section{Highly transitive actions.}
Let $G$ be a Tarski monster. We argue by induction on $k$ that the action of $\Gamma = \Out(F_n)$ on $\bar{V}_n(G)$ is $k$-transitive. 

\subsection{The induction basis}
\begin{proposition}  \label{prop:ind_basis}
For every $n \geq 3$, $\Aut(F_n)$ acts transitively on $V_n(G)$,
where $G$ is a Tarski monster group.
\end{proposition}
\begin{proof}
Let us fix a basepoint $\g = (g_1,g_2,\ldots,g_n) \in V_n(G)$. We
are at liberty to choose a convenient base point - and using Lemma
\ref{lem:spread} repeatedly we impose the condition $\langle g_i,g_j
\rangle = G \ \forall i \ne j$. Now given any $\h = (h_1,h_2,\ldots,
h_n) \in V_n(G)$ we have to exhibit a sequence of Nielsen
transformations taking $\g$ to $\h$. Since any generating set
contains a generating pair we may assume, after renumbering the
indices, that $\langle h_1,h_2 \rangle = G$. Since $G$ has spread
greater or equal to $2$ there exists some $z \in G$ such that
\begin{equation} \label{eqn:z_ass}
\langle g_2,z \rangle = \langle h_1,z \rangle = G.
\end{equation}
We proceed with the following sequence of Nielsen transformations:
\begin{equation} \label{eqn:seq_trans} \g =
\begin{pmatrix} g_1 \\ g_2 \\ g_3 \\ \vdots \\ g_n \end{pmatrix}
\looparrowright \begin{pmatrix} g_1 \\ g_2 \\ z \\ \vdots \\ g_n \end{pmatrix}
\looparrowright \begin{pmatrix} h_1 \\ g_2 \\ z \\ \vdots \\ g_n \end{pmatrix}
\looparrowright \begin{pmatrix} h_1 \\ h_2 \\ z \\ \vdots \\ g_n \end{pmatrix}
\looparrowright \begin{pmatrix} h_1 \\ h_2 \\ h_3 \\ \vdots \\ h_n \end{pmatrix}
= \h.
\end{equation}
The first $\looparrowright$ stands for a sequence of Nielsen transformations of the form $w(R_{3,1},R_{3,2})$ where $w$ is any free word on two generators satisfying $w(g_1,g_2) = g_3^{-1}z$. The existence of such a word is guaranteed by the fact that $\langle g_1,g_2 \rangle = G$. The next three $\looparrowright$ use the exact same argument, but instead of using the fact that $\langle g_1, g_2 \rangle = G$ they appeal in turn to the two parts of Equation (\ref{eqn:z_ass}) and then to our assumption that $\langle h_1,h_2 \rangle = G$.
\end{proof}
Since $\Aut(F_n) \curvearrowright V_n(G)$ is transitive, so is the
quotient action  $\Gamma \curvearrowright \bar{V}_n(G)$, which is
exactly what we require for the basis of our induction.

\subsection{General Tarski Monsters}
We choose a base $k$-tuple $$\left([\g^1],[\g^2],[\g^3], \ldots, [\g] = [\g^k] \right)$$ of distinct elements in $\bar{V}_n(G)$. In order to establish the induction step we have to show that for any $[\h] \not \in \{[\g^1],[\g^2],[\g^3], \ldots, [\g^{k-1}]\}$ there is a group element $\gamma \in \Gamma$ such that $\gamma [\g^i] = [\g^i], \ \forall 1 \le i \le k-1$ and $\gamma [\g] = [\h]$.

Again we have a lot of freedom in the choice of our basis $k$-tuple. We make our choice by picking a matrix
$$
A = \begin{pmatrix}
g_{1}^{1}  & \ldots & g_{1}^{k} \\
g_{2}^{1} & \ddots  & g_{2}^{k} \\
\vdots & &  \vdots \\
g_{n}^{1} & \ldots & g_{n}^{k}
\end{pmatrix} =
\begin{pmatrix}
-  & \g_1 & - \\
- & \g_2 & - \\
  & \vdots  &  \\
- & \g_n & - \\
\end{pmatrix} =
\begin{pmatrix}
|  &  & | \\
\g^{1} & \hdots  & \g^{k} \\
| & & |
\end{pmatrix},
$$
satisfying the conditions guaranteed by Lemma \ref{lem:tech_3_gen},
and taking its columns as representatives. The element $\h =
(h_1,h_2,\ldots,h_n) \in V_n(G)$ on the other hand is dictated to
us. But, since $\langle \h \rangle = G$ we may assume, after
possibly reordering the indices that $\langle h_1,h_2 \rangle = G$.

We wish to proceed in much the same way as we did in Equation (\ref{eqn:seq_trans}), taking $\g$ to $\h$, but this time we have to be careful not to touch the elements $\{\g^i \ | \ 1 \le i \le k-1\}$. Let us find an element $z \in G$ such that all the $3 \times k$ matrices $A_2,A_3,A_4$ below satisfy the equivalent conditions of Proposition \ref{prop:gen_prod}
\begin{eqnarray*} \label{eqn:z_ass_3}
A_2 & = &
\begin{pmatrix}
-  & - & \g_2 & - & - \\
-  & - & \g_3 & - & - \\
 g^1_4 & g^2_4 & \ldots & g^{k-1}_4 & z
\end{pmatrix}  \\ \\
A_3 & = &
\begin{pmatrix}
g^1_1  & g^2_1 & \ldots & g^{k-1}_1 & h_1 \\
 - & - & \g_3 & - & - \\
 g^1_4 & g^2_4 & \ldots & g^{k-1}_4 & z
\end{pmatrix}  \\ \\
A_4 & = &
\begin{pmatrix}
g^1_1  & g^2_1 & \ldots & g^{k-1}_1 & h_1 \\
g^1_2  & g^2_2 & \ldots & g^{k-1}_2 & h_2 \\
g^1_4 & g^2_4 & \ldots & g^{k-1}_4 & z \\
\end{pmatrix}
\end{eqnarray*}
For the matrices $A_2, A_4$ this can be achieved by avoiding
finitely many bad values of $z$, by an argument identical to that
used in the proof of Lemma \ref{lem:tech_3_gen}. For the matrix
$A_3$ we can use the same argument assuming $\langle h_1, g_3
\rangle = G$. On the other hand if $\langle h_1,g_3 \rangle$ is a
cyclic subgroup then the desired condition $[h_1,g_3,z] \ne
[g^i_1,g^i_3,g^i_4] \ \ \forall 1 \le i < k$ is satisfied
automatically, for any choice of $z$ which makes $(h_1,g_3,z)$ into
a generating set, just because all the other columns satisfy the
condition that every pair of elements generates $G$ and the last
column does not.

We now proceed by applying the following sequence of Nielsen transformations:
$$\g =
\begin{pmatrix} g_1 \\ g_2 \\ g_3 \\ g_4 \\ \vdots \\ g_n \end{pmatrix}
\looparrowright \begin{pmatrix} g_1 \\ g_2 \\ g_3 \\ z \\ \vdots \\ g_n \end{pmatrix}
\looparrowright \begin{pmatrix} h_1 \\ g_2 \\ g_3 \\ z \\ \vdots \\ g_n \end{pmatrix}
\looparrowright \begin{pmatrix} h_1 \\ h_2 \\ g_3 \\ z \\ \vdots \\ g_n \end{pmatrix}
\looparrowright \begin{pmatrix} h_1 \\ h_2 \\ h_3 \\ z \\ \vdots \\ g_n \end{pmatrix}
\looparrowright \begin{pmatrix} h_1 \\ h_2 \\ h_3 \\ h_4 \\ \vdots \\ h_n \end{pmatrix}
= \h.$$
This time the first $\looparrowright$ corresponds to $w(R_{4,1},R_{4,2},R_{4,3})$, where $w$ is chosen so as to satisfy the two conditions:
\begin{itemize}
\item $w(g_1,g_2,g_3) = g_4^{-1}z$,
\item $w(g^i_1,g^i_2,g^{i}_3) = \id, \quad \forall 1 \le i < k$.
\end{itemize}
The existence of such a word follows directly from condition (\ref{itm:kernels}) of Proposition \ref{prop:gen_prod}, combined with our assumptions on the matrix $A$. The existence of Nielsen transformations realizing the next three $\looparrowright$, while fixing $\g^i \ \ \forall 1 \le i < k$ follows, in the exact same way using our assumptions on the matrices $A_2,A_3,A_4$ respectively.

The last $\looparrowright$ can be treated in a similar fashion as long as the matrix:
\begin{eqnarray*} \label{eqn:z_ass_2}
A_5 =
\begin{pmatrix}
g^1_1  & g^2_1 & \ldots & g^{k-1}_1 & h_1 \\
g^1_2  & g^2_2 & \ldots & g^{k-1}_2 & h_2 \\
g^1_3 & g^2_3 & \ldots & g^{k-1}_3 & h_3 \\
\end{pmatrix}
\end{eqnarray*}
satisfies the conditions of Proposition \ref{prop:gen_prod}.

Note that if there is a pair of indices such that $\langle h_i,h_j \rangle \ne G$ these conditions are automatically satisfied, perhaps after rearrangement of the indices. Indeed after rearranging the indices so that $$\langle h_2,h_3 \rangle \ne G$$ the last column is the only one in $A_5$ with this property and thus it cannot be in the $\Aut(G)$ orbit of any of the other columns. Thus we can assume that every pair of $h$ entries generates the entire group.

Finally if the conditions of Proposition \ref{prop:gen_prod} are not
satisfied for $A_5$. Not even after we change the order of the
indices, by making arbitrary permutations of the first four rows of
$A$. This means that there are four automorphisms $\sigma,
\tau,\eta, \theta \in \Aut(G)$ and four columns $1 \le i < j < l < m
\le k$ such that the corresponding $4 \times 4$ minor admits the
forbidden configuration described in condition (\ref{itm:config}) of
Lemma \ref{lem:tech_3_gen}$$
\begin{pmatrix}
g^i_1 & g^j_1 & g^l_1 & g^m_1 \\
g^i_2 & g^j_2 & g^l_2 & g^m_2 \\
g^i_3 & g^j_3 & g^l_3 & g^m_3 \\
g^i_4 & g^j_4 & g^l_4 & g^m_4
\end{pmatrix} =
\begin{pmatrix}
\cdot & (h_1)^{\tau} & (h_1)^{\eta} & (h_1)^{\theta} \\
(h_2)^{\sigma} & \cdot & (h_2)^{\eta} & (h_2)^{\theta} \\
(h_3)^{\sigma} & (h_3)^{\tau} & \cdot & (h_3)^{\theta} \\
(h_4)^{\sigma} & (h_4)^{\tau} & (h_4)^{\eta} & \cdot
\end{pmatrix},
$$
contradicting our construction of the matrix $A$.

\section{Concerning faithfulness of the action} \label{sec:faithful}
\subsection{A theorem of Magnus}
We will require the following, well known, theorem of Wilhelm Magnus from 1930:
\begin{theorem} \cite[Theorem N5 page
172]{MKS:Combinatorial_group_theory}.
Let $a,b \in F_m$ be elements of the free group and assume that they generate the same normal subgroup
$$\langle a \rangle^{F_m} = \langle b \rangle^{F_m},$$
then $a = g b^{\epsilon} g^{-1}$ for some $g \in F_m$ and some $\epsilon \in \{\pm 1\}$.
\end{theorem}
\noindent Namely the normal closure of a cyclic group, determines the generator up to conjugation and inversion.

\subsection{The action $\Out(F_n)$ on the redundant locus of $F_{n-1}$}
\begin{definition}
If $n > d(G)$ the redundant locus of $V_n(G)$ is defined as:
$$R_n(G) = \left \{\phi \in V_n(G) \ | \ \left \langle \phi(x_1),\ldots,\phi(x_{n-1} )\right \rangle = G, {\text{ for some basis }} x_1,\ldots,x_n < F_n \right\}$$
and $\bar{R}_n(G) < \bar{V}_n(G)$ is the image of this (invariant) set, modulo $\Aut(G)$.
\end{definition}

\begin{lemma} \label{lem:faithful_Rn}
For every $n \ge 3$, $\Out(F_n)$ acts faithfully on $\bar{R}_n(F_{n-1})$.
\end{lemma}
\begin{proof}
It is easy to verify that $\phi: F_n \arrow F_{n-1}$ corresponds to an element of $R_{n}(F_{n-1})$ if and only if it is surjective and the kernel is generated, as a normal subgroup, by a primitive element of $F_n$.

Let $\alpha \in \Aut(F_n)$ represent an element of $\Out(F_n)$ that is in the kernel of the action on $\bar{R}_n(F_{n-1})$. Thus by definition for every $\phi \in R_{n}(F_{n-1})$ there exists some $\sigma \in \Aut(F_{n-1})$ such that $\phi \circ \alpha = \sigma \circ \phi$, and in particular $$\alpha^{-1}(\ker(\phi)) = \ker (\phi \circ \alpha)  = \ker(\sigma \circ \phi) = \ker(\phi).$$
So $\alpha$ acts trivially on the collection of normal subgroups generated by a primitive element. By Magnus' theorem stated above, for every primitive element $x \in F_n$ there exist $g_x \in F_n, \epsilon_x \in \{\pm 1\}$ such that
\begin{equation} \label{eqn:alpha}
\alpha x = g_x x^{\epsilon} g_x^{-1}.
\end{equation}
In particular setting $g_i = g_{x_i}, \epsilon_i = \epsilon_{x_i}$ for some basis $\x = \{x_1,x_2,\ldots,x_n\}$ we have:
$$\alpha(\x) =
\begin{pmatrix}
\alpha_1 (\x) \\
\alpha_2 (\x) \\
\vdots \\
\alpha_n  (\x) \\
\end{pmatrix} =
\begin{pmatrix}
x_1^{\epsilon_1} \\
g_2 x_2^{\epsilon_2} g_2^{-1} \\
\vdots \\
g_n x_n^{\epsilon_n} g_n^{-1} \\
\end{pmatrix}.$$
Here, since to begin with, $\alpha$ is defined only up to inner automorphisms, we have assumed without loss of generality that $g_1=1$.

Let $Y$ be the $2n$-regular Cayley tree of $F_n$ with respect to the given set of generators. We label the oriented edges of $Y$ by elements of the corresponding symmetric generating set $\{x_1, x_1^{-1},x_2, \ldots, x_n^{-1}\}$ in such a way that $F_n$ is identified with the group of color preserving automorphisms of the tree. Every element $x \in F_n$ acts on $Y$ as a hyperbolic automorphism with translation length $\ell(x) := \min \{d(v,xv) \ | \ v \in Y \} > 0$ and axis $X_x = \{v \in Y \ | \ d(v,xv) = \ell(x)\}$. Equation (\ref{eqn:alpha}) implies that $\alpha(x_i)$ is a hyperbolic element of translation length $1$ whose axis is $g_i X_{x_i}$ (with either of its two possible orientations). In particular the axis of $\alpha(x_i)$ is labeled either $\ldots x_i,x_i,x_i, \ldots$ or with the inverse of this sequence. Similarly the translation length and axis labeling is preserved for any primitive element of $\Aut(F_n)$.

We claim that $X_{\alpha(x_i)} \cap X_{\alpha(x_j)} =  g_i X_i \cap g_j X_j \ne \emptyset, \quad \forall i \ne j.$ Indeed, assuming the contrary, since $\alpha$ preserves the translation length of primitive elements we have
$$2 = \ell(x_ix_j) = \ell (\alpha(x_i) \alpha(x_j)) = 2 + 2 d(X_{\alpha(x_i)}, X_{\alpha(x_j)});$$ so $d(X_{\alpha(x_i)}, X_{\alpha(x_j)}) = 0$ which is a contradiction. Now, by the version of Helly's theorem for trees (see for example \cite[I.6.5 Lemma 10]{Serre:Trees}) this implies that there is a point $o \in \cap_{i = 1 \ldots n} X_{\alpha(x_i)}$. After conjugation by an appropriate power of $x_1 = \alpha(x_1)$ we may assume that
$$o = \cap_{i = 1 \ldots n} X_{\alpha(x_i)} = \cap_{i = 1 \ldots n} X_{i}.$$
But now $X_i = X_{\alpha(x_i)}$ since these two axes share a point and, up to orientation, they have the same coloring; hence upon replacing $g_i$ by $g_i x_i^{m_i}$ for an appropriate choice of power $m_i$, we can assume that $g_i o = o$. This immediately implies that $g_i = 1 \quad \forall i$.

We still have to show that $\epsilon_i = 1 \quad \forall i$. But if, say, $\epsilon_1 = -1$ then the primitive element $x_1x_2x_3$ will map to $x_1^{-1} x_2^{\epsilon_2} x_3^{\epsilon_3}$. A short verification will show that, regardless of the values of $\epsilon_2,\epsilon_3$, this element is neither conjugate to $x_1x_2x_3$ nor to $(x_1x_2x_3)^{-1}$. This completes the proof of the lemma.
\end{proof}

\subsection{Group laws on two letters are universal}
\label{sec:2_univ}
\begin{lemma}
Assume that $G$ is a finitely generated group that satisfies a group law. Then $G$ already satisfies a non-trivial group law on two letters.
\end{lemma}
\begin{proof}
Assume that $G$ satisfies a group law $w \in F_m$ for some $m > 2$. By \cite[Corollary 3.3]{BG:Limit} there is a homomorphism $\phi = (\phi_1,\phi_2, \ldots, \phi_n): F_m \arrow F_2$ such that $\phi(w) \ne 1$. Now $G$ satisfies the non-trivial group law $\phi(w)$ as
$$\phi(w)(g,h) = w(\phi_1(g,h),\phi_2(g,h), \ldots, \phi_n(g,h)) = 1, \quad \forall g,h \in G.$$
\end{proof}

\subsection{When $G$ satisfies a group law the action is not faithful}
\begin{proposition} \label{prop:if}
If $G$ is any finitely generated group which satisfies a group law, then for any $n>\max\{2,d(G)\}$, the action of $\Out(F_n)$ on $\bar{V}_n(G)$ is not faithful.
\end{proposition}
\begin{proof}
As we saw in Section \ref{sec:2_univ} we may assume that $G$ satisfies a group law on two letters, i.e. there exists a word $w \in F_2$ such that $w(g,h) = 1, \ \ \forall g,h \in G$.

Consider the automorphism $\alpha = w(R_{n,1},R_{n,2})$ given explicitly by
\[
    \alpha(x_1,\dots,x_n) =
    (x_1,\dots,x_{n-1}, x_n \cdot w(x_1,x_2) ).
\]
$\alpha$ is not the trivial automorphism because, by assumption $w$ is a non-trivial word in $F_2$ and $R_{n,1}, R_{n,2}$ generate a free subgroup of $\Out(F_n)$. But by construction $\alpha$ acts trivially on $G^n$.
\end{proof}

\subsection{When the action is not faithful $G$ satisfies a group law}
The following completes the proof of our main theorem.
\begin{proposition}
Let $G$ be a Tarski monster and $n \ge 3$. Then the action of $\Out(F_n)$ on $\bar{V}_n(G)$ is faithful if and only if $G$ satisfies no group law.
\end{proposition}
\begin{proof}
The only if part is treated in Proposition \ref{prop:if}.
Assume that $G$ is a Tarski monster group and that the action of $\Out(F_n)$ on $\bar{V}_n(G)$ is not faithful. We will show that $G$ satisfies a group law.

It is enough to show that there exists some non-trivial word $w=w(x,y) \in F_2$ such that $w(a,b)=1$ for any $(a,b) \in V_2(G)$. Indeed, assume that $w$ can be written in reduced form as $w=z_1 \dots z_n$,
where $z_1,\dots,z_n \in \{x^{\pm 1},y^{\pm 1}\}$, and let
\[
   v = wzw^{-1}z^{-1} = z_1 \dots z_n z z_n^{-1} \dots z_1^{-1} z^{-1},
\]
where $z \in  \{x^{\pm 1},y^{\pm 1}\}$ satisfies that $z \neq
z_1^{-1},z_n,z_n^{-1}$. Then $v=v(x,y)$ is a non-trivial word in
$F_2$. Moreover, if $(a,b) \in V_2(G)$ then $v(a,b)=1$ since $w(a,b)=1$. If $a$ and $b$ do not generate $G$, then they belong to the same cyclic group, and so there exist some $c \in G$ and $i,j \in \Z$ s.t. $a=c^i$ and $b=c^j$. In this case, $w(a,b)=w(c^i,c^j)=c^k$ for
some $k \in \Z$, and so $v(a,b)=v(c^i,c^j)=c^kc^lc^{-k}c^{-l}=1$ (where $l \in \{ \pm i, \pm j \}$). Hence, $G$ satisfies a group law with the word $v$.

By assumption, there exists some automorphism
$\alpha=(\alpha_1,\alpha_2,\ldots,\alpha_n) \in \Aut(F_n)$, which is not an inner automorphism, such that for any $\g \in V_n(G)$
there exists some $\sigma \in \Aut(G)$ such that.
\begin{align*}
\alpha_1(\g) &= \sigma(g_1), \\
\alpha_2(\g) &= \sigma(g_2), \\
& \vdots \\
\alpha_n(\g) &= \sigma(g_n).
\end{align*}

In particular for any $(a,b) \in V_2(G)$ and for any word $u \in F_2$, we can apply the above to the the $n$-tuple $(a,b,u(a,b), 1,1,\ldots) \in V_n(G)$ obtaining the equation
$$\alpha_3(a,b,u(a,b)) = \sigma(u(a,b)) = u(\sigma(a),\sigma(b)) = u\bigl(\alpha_1(a,b,u(a,b)), \alpha_2(a,b,u(a,b))\bigr).$$ Here we used $\alpha_1(a,b,u(a,b))$ as a short for $\alpha_1(a,b,u(a,b),1,1,\ldots)$.

Now consider the three words on two letters
\begin{eqnarray*}
    w_3^u(x,y) & = & u  \left( \alpha_1(x,y,u(x,y)), \alpha_2(x,y,u(x,y)) \right) \cdot  \alpha_3(x,y,u(x,y))^{-1} , \\
    w_2^u(x,y) & = & u  \left( \alpha_1(x,u(x,y),y), \alpha_3(x,u(x,y),y) \right) \cdot \alpha_2(x,u(x,y),y)^{-1} , \\
    w_1^u(x,y) & = & u  \left( \alpha_2(u(x,y),x,y), \alpha_3(u(x,y),x,y) \right) \cdot  \alpha_1(u(x,y),x,y)^{-1} .
\end{eqnarray*}
By permuting the role of $a,b,u(a,b)$ among the first three coordinates in the above argument we know that
$w_i^u(a,b) = 1 \quad \forall (a,b) \in V_2(G), \forall u \in F_2(x,y).$ If one of the $w_i^u$ is a non-trivial word in $F_2$ then we have our group law. But if $w_i^u$ represents the trivial word in $F_2$ for every $i = 1,2,3$ and every $u \in F_2$ then $\alpha=(\alpha_1,\alpha_2,\alpha_3)$ gives rise to an element of $\Out(F_3)$ in the kernel of the action on $\bar{R}_3(F_2)$ and hence a contradiction to Lemma \ref{lem:faithful_Rn}.

\end{proof}

\section{Final remarks and questions}
We conclude with a remark and a few open questions

\begin{remark}
The proof of Theorem \ref{thm:Main} may be somewhat simplified if we assume that $G$ is a Tarski monster with $\Out(G) = \trivgp$; Indeed given any countable group $H$, the existence of Tarski monsters with $\Out(G) = H$  is guaranteed by \cite{Obraztsov:Out_Tarski}. 
 \end{remark}
 
\subsection{Lower rank groups}
\begin{question}
What about $\Out(F_2)$ and $\Out(F_3)$? Do they admit a highly
transitive action on a set?
\end{question}
There is a chance that the action that we study in this paper, of
$\Out(F_n)$ on $\bar{V}_n(G)$ still has very good transitivity
properties for $n=3$. Even if this is true it seems that the proof
would be much harder as it would require a much better understanding of the Tarski monster $G$ and its automorphisms. A topic that we carefully avoided in this paper. The main obstacle is to find Tarski monster groups that satisfy some $3 \times 3$ analogue for condition \ref{itm:config} of Lemma \ref{lem:tech_3_gen}.

On the other hand it is plausible that one can construct completely different actions in this lower rank setting. For example since $\Out(F_2) = \PSL_2(\Z) = \Z/2\Z * \Z/3\Z$, there is a good chance that one can construct a highly transitive action for this group using the methods of \cite{Dixon:free_HT}. It is important to note though that this is not automatic. It is not in general true that if a finite index subgroup admits a highly transitive action then so does the group itself. Even though the other direction is true of course.

As for $n = 3$. At least for the group $\Aut(F_3)$ it follows from \cite[Corollary 1.2]{GL:lin_rep_Aut_Fn} that there is a finite index subgroup $\Delta < \Aut(F_3)$ that maps onto a finitely generated free group, and hence admits a (non-faithful) highly transitive action. Again it is plausible that one might be able to construct a highly transitive non-faithful action of $\Aut(F_3)$ from this.

\subsection{Mapping class groups}
\begin{question}
Does $\operatorname{Mod}(S)$, the mapping class group of a closed orientable surface $S$ of a high enough genus admit a highly transitive action on a set? If not, how about $k$-transitive actions for various values of $k$?
\end{question}
Recall that from \cite{GG:Primitive} it follows that these groups do admit faithful primitive actions.

\subsection{Faithfulness}
Our proof of the faithfulness statement in section \ref{sec:faithful} gives rise to the following general theorem:
\begin{theorem}
For any finitely generated group $G$, the following are equivalent:
\begin{itemize}
\item The action of $\Out(F_n)$ on the $\Aut(G)$-classes of $\Hom(F_n,G)$ is faithful for all large enough $n$.
\item The group $G$ does not satisfy a group law.
\end{itemize}
\end{theorem}

It is a very natural question if the above still holds if one
replaces $\Hom(F_n, G)$ by $\bar{V}_n(G)$.


\bibliography{../tex_utils/yair}
\medskip

\end{document}